\documentclass[11pt]{amsart}

\usepackage{amsmath,amsfonts,latexsym}
\usepackage{amscd,amssymb,amsmath}
\usepackage{amsbsy}
\usepackage[all,cmtip]{xy}

\usepackage{color}
\newcommand{\todo}[1]{{\color{red} #1}}

\newcommand{\bea}{\begin{eqnarray*}}
\newcommand{\eea}{\end{eqnarray*}}
\newcommand{\zz}[1]{}

\newcommand{\wh}{\widetilde H^*}

\newtheorem{theorem}{Theorem}[section]
\newtheorem{proposition}[theorem]{Proposition}
\newtheorem{corollary}[theorem]{Corollary}
\newtheorem{remark}[theorem]{Remark}

\newtheorem{definition}[theorem]{Definition}
\newtheorem{lemma}[theorem]{Lemma}
\newtheorem{example}[theorem]{Example}

\newtheorem{fact}[theorem]{Fact}

\newcommand{\UU}{\mathcal U}
\newcommand{\VV}{\mathcal V}

 \newcommand{\NN}{{\mathbb{N}}}
 \newcommand{\ZZ}{{\mathbb{Z}}}
 \newcommand{\RR}{{\mathbb{R}}}
 \newcommand{\CC}{{\mathbb{C}}}
 \newcommand{\HH}{{\mathbb{H}}}
 \newcommand{\w}{{\mathrm{w}}}

\newcommand{\sct}{{\rm sct}}
\newcommand{\ct}{{\rm ct}}
\newcommand{\wct}{{\rm wct}}
\newcommand{\swct}{{\rm swct}}
\newcommand{\cat}{{\rm cat}}
\newcommand{\hdim}{{\rm hdim}}

\newcommand{\cwgt}{\mathrm{cwgt}}
\newcommand{\swgt}{\mathrm{swgt}}

\begin{document}
\thanks{$^{*}$ Supported by the Slovenian Research Agency program P1-0292 and grant N1-0083}
\thanks{$^{**}$ Supported by the Polish  Research Grant NCN   Sheng 1 UMO-2018/30/Q/ST1/00228 }
\thanks{$^{***}$ Supported by the Slovenian Research Agency program P1-0292 and grants N1-0083,
N1-0064}
\title{Estimates of covering type and minimal triangulations based on category weight}
\author[Dejan Govc]{Dejan Govc$^*$}
\author[Wac{\l}aw Marzantowicz]{Wac{\l}aw Marzantowicz$^{**}$}
\author[Petar Pave\v{s}i\'{c}]{Petar Pave\v si\'{c}$^{***}$}

\address{$^*$, $^{***}$ Faculty of Mathematics and Physics, University of Ljubljana,
Jadranska 21,  1000 Ljubljana, Slovenija}
\email{dejan.govc@gmail.com, petar.pavesic@fmf.uni-lj.si}
\address{$^{**}$ \;Faculty of Mathematics and Computer Science, Adam Mickiewicz University of 
Pozna{\'n}, ul. Uniwersytetu Pozna\'nskiego 4, 61-614 Pozna{\'n}, Poland.}
 \email{marzan@amu.edu.pl}


\begin{abstract}
In a recent publication \cite{GovcMarzantowiczPavesic2019} we have introduced a new method, based on
the Lusternik-Schnirelmann category and the cohomology ring of a space $X$, that yields lower bounds
for the size of a triangulation of  $X$. In this paper we present an important extension
that takes into account the fundamental group of $X$. In fact, if $\pi_1(X)$ contains elements of
finite order, then one can often find cohomology classes of high 'category weight', which in turn
allow for much stronger estimates of the size of triangulations of $X$. We develop several weighted
estimates and then apply our method to compute explicit lower bounds for the size of triangulations 
of orbit spaces
of cyclic group actions on a variety of spaces including products of spheres, Stiefel manifolds,
Lie groups and highly-connected manifolds.
\end{abstract}

\maketitle

\section{Introduction}
\label{sec:Introduction}

A \emph{triangulation} of a topological space $X$ is a simplicial complex $K$ together with a
homeomorphism $X\approx |K|$ between $X$ and the geometric realization of $K$. Clearly, a 
triangulable
space admits a triangulation by a finite complex if, and only if it is compact.

Given a compact, triangulable space $X$, let $\Delta(X)$ denote the minimal number of vertices in
a trian\-gu\-lation of $X$. In spite of recent spectacular advances on the number of simplices
that are needed to triangulate a given manifold by Adiprasito \cite{Adiprasito}, see also
Klee-Nowik \cite{KN}, computation of $\Delta(X)$ remains
a hard question, because there are no standard constructions for triangulations with few vertices of 
a given space, nor are there sufficiently general methods to prove that some specific triangulation
is in fact minimal. Note that the estimates of [1], [20], and other papers referenced therein 
require as input the values of $\Delta(X)$ and do not discuss how to derive it.

Even when there is an explicit triangulation at hand it may be difficult to show
that it represents a specific manifold. This was the case of the Brehm-K\"uhnel triangulation
\cite{BrehmKuhnel1992} whose cohomology is that of the quaternionic projective plane, but it 
required a
very hard computation with combinatorial Pontrjagin classes by Gorodkov \cite{Gorodkov2019} to show 
that
it is indeed the minimal triangulation of $\HH P^2$. This explains why apart from classical minimal
triangulations of spheres and closed surfaces, and a special family of minimal triangulations
for certain sphere bundles over a circle (so called Cs\'asz\'ar tori - see K\"uhnel \cite{K}),
there exists only a
handful of examples for which the minimal triangulations are known. An exhaustive
survey of the results and the existing literature on this problem can be found in Lutz \cite{Lutz}.

In \cite{GovcMarzantowiczPavesic2019} we introduced several new estimates for the minimal number of 
vertices that
are needed to triangulate a compact triangulable space $X$ based on the Lusternik-Schnirelmann 
category of
$X$ and on the structure of the cohomology ring of $X$. In the case of manifolds these estimates 
were
improved by using information obtained from the fundamental group in \cite{Pavesic2019} and on the
Lower Bound Theorem in \cite{GovcMarzantowiczPavesic2020}. The information on the number,  or 
respectively
rate of growth of number of simplices  included in the latter has been used in   \cite{Knudson} and
\cite{ScoccolaPerea} respectively. If the fundamental group of $X$ contains elements of finite 
order, then
certain cohomological properties of finite groups are reflected in the cohomology of $X$, which 
allows to
identify elements of high \emph{category weight}. This idea was first exploited by
Fadell and Husseini \cite{FadellHusseini1992}, and was further developed by Rudyak \cite{Rudyak} and
Strom \cite{Strom}. The main results of this paper are Theorems \ref{thm:weighted estimate} and
\ref{thm:swct estimate} that lead to  estimates of the size of triangulations of $X$ based on the cohomology
weight of $X$ (Corollary \ref{cor:Delta by wct}). Although our method is quite general, our main applications 
in this paper are estimates
of the size of triangulations of several classes of closed manifolds with finite cyclic fundamental group.

{\bf Outline.} In Section \ref{sec:Homotopy triangulations and category weight} we give a brief 
overview of our basic technical tools, covering type of a space and category weight of a cohomology
class. Specifically, in \ref{subsec:Covering type and homotopy triangulations} 
we review the concepts of \emph{triangulation size} $\Delta(X)$, \emph{strict covering type} 
$\sct(X)$ and \emph{covering type} $\ct(X)$ of a space $X$, and state the main relation 
$\Delta(X)\ge \sct(X)\ge\ct(X)$.
Then in \ref{subsec:Category weight} we recall the definition of \emph{category weight} $\cwgt(u)$ 
and of \emph{strict category weight} $\swgt(u)$ of a cohomology class $u$. The first concept 
is more geometric while the second is homotopy invariant and they are related by 
$1\le\swgt(u)\le \cwgt(u)\le |u|$.
Both weights are hard to compute exactly, so in Section \ref{sec:Category weight 
estimates} we first introduce a \emph{weight estimator} ${\bf w}$, designed as a readily 
computable lower bound for category weights. Finally, we define $\wct(s;{\bf w})$, which estimates 
covering type (relative to ${\bf w}$) through a specific sequence of cohomology classes $s$,
and $\swct(X)$ as a maximum of $\wct(s;{\bf w})$ over all sequences $s$. $\wct(s;{\bf w})$ and 
$\swct(X)$ play a central role in our main Theorems \ref{thm:weighted estimate} and \ref{thm:swct 
estimate} and in ensuing computations.  
The final Section \ref{sec:Manifolds with finite cyclic fundamental group} is divided into several
subsections, in which we give a systematic study of various classes of closed manifolds with 
a finite cyclic fundamental group.

\section{Homotopy triangulations and category weight}
\label{sec:Homotopy triangulations and category weight}

In this section we review two concepts that will serve as a base for our computations.

\subsection{Covering type and homotopy triangulations}
\label{subsec:Covering type and homotopy triangulations}

Our estimates of the minimal number of vertices $\Delta(X)$ in a triangulation of $X$ passes through an
intermediate concept called \emph{covering type} introduced by Karoubi and Weibel \cite{KaroubiWeibel2016}
that we briefly recall. An open  cover $\mathcal{U}$ of $X$ is said to be a \emph{good cover} if
all elements of $\mathcal{U}$ and all their  non-empty finite intersections are contractible.
Standard examples of good covers are covers of Riemannian manifolds by geodesically convex open
balls and covers of polyhedra by open stars of vertices with respect to some triangulation.

Karoubi and Weibel defined $\sct(X)$, the \emph{strict covering type} of $X$ as the minimal cardinality
of a good cover for $X$. Clearly $\Delta(X)\ge \sct(X)$, whenever $X$ can be triangulated.
Note that strict covering type can be infinite (e.g.,
if $X$ is an infinite discrete space) or even undefined, if the space does not admit any good covers
(e.g. the Hawaiian earring).
In what follows we will always assume that the spaces under consideration admit finite good covers.

There are simple examples of homotopy equivalent spaces for which minimal good covers have
different cardinality, showing that the notion is not homotopy invariant. Karoubi and Weibel defined
the \emph{covering type} of $X$ as the minimal value of $\sct(Y)$
among all spaces $Y$ that are homotopy equivalent to $X$.
The covering type is obviously a homotopy invariant of the space and can thus be related to
other homotopy invariants - cf. \cite{KaroubiWeibel2016} and \cite{GovcMarzantowiczPavesic2019}.

The main feature of good covers is that the pattern of intersections of sets of a good cover capture
the homotopy type of a space. In fact, let $N(\UU)$ denote the \emph{nerve} of the open
cover  $\UU$ of $X$, and let $|N(\UU)|$ be its geometric realization. We may identify the vertices of $|N(\UU)|$ with the elements of $\UU$ and the points of
$|N(\UU)|$ with the convex combinations of elements of $\UU$.
If $\UU$ is \emph{numerable}, that is, if $\UU$ admits a subordinated partition of unity $\{\varphi_{_U}\colon X\to [0,1]\mid U\in\UU\}$,
then the formula
$$\varphi(x):=\sum_{U\in\UU} \varphi_{_U}(x)\cdot U$$
determines the so called \emph{Aleksandroff map} $\varphi\colon X\to |N(\UU)|$, which has many remarkable
properties. In particular we have the following classical result, whose discovery is variously attributed to
J.~Leray, K.~Borsuk and A.~Weil (see \cite[Corollary 4G.3]{Hatcher2002} for a modern proof).

\begin{theorem}[Nerve Theorem]
\label{thm:Nerve}
If  $\UU$ is a numerable good cover of $X$, then the Aleksandroff map $\varphi: X \to |N(\UU)|$ is a homotopy equivalence.
\end{theorem}

We will use repeatedly the following immediate consequences of the theorem:
\begin{itemize}
\item
If $U_0,\ldots,U_n$ are elements of a good cover $\UU$, whose intersection $U_0\cap\ldots\cap U_n$
is non-empty, then their union  $U_0\cup\ldots\cup U_n$ is contractible (in fact, it is homotopy
equivalent to a simplex).
\item
If $\UU$ is a good cover of a space $X$ whose homotopy dimension is $\hdim(X)=n$, then there exists
at least $n+1$ elements of $\UU$ whose intersection is non-empty (otherwise the homotopy dimension
of $|N(\UU)|$ would be strictly smaller than $n$).
\end{itemize}

Since every polyhedron admits a good cover by open stars of its vertices one may restate the
Nerve Theorem
by saying that the nerve of a good cover of $X$ provides a \emph{homotopy triangulation} of $X$.
Therefore,
the covering type $\ct(X)$ provides a homotopy invariant lower bound for the number of vertices of a minimal
triangulation of $X$:
\begin{equation}\label{Delta}
\Delta(X)\ge \sct(X)\ge\ct(X)\,.
\end{equation}
This relation is the base for the applications of homotopical methods in the
computation of $\Delta(X)$.

\subsection{Category weight}
\label{subsec:Category weight}

We shall use the notions {\bf category weight} introduced by E. Fadell and S. Husseini in
\cite{FadellHusseini1992}, and {\bf strict category weight}, introduced by Y. Rudyak in \cite{Rudyak}, and
independently by J. Strom in \cite{Strom} (who called it \emph{essential category}).
For convenience of the reader, we recall the definitions and basic properties of these notions.

Let $X$ be a path-connected space having the homotopy type of a CW-complex.
A subset $A\subseteq X$ is said to be \emph{categorical} in $X$ if the inclusion
$\iota\colon A\hookrightarrow X$ is null-homotopic. The (Lusternik-Schnirelmann) \emph{category} of
$X$, denoted $\cat(X)$, is the minimal integer $k$, such that $X$ can be covered by $k$ open
categorical subsets (note that much of the literature on LS-category uses the normalized version
of LS-category, which is by one smaller than ours, cf \cite{CLOT}).
For a commutative and unital ring $R$ let $\wh(X;R)$ denote the reduced cohomology  ring of $X$
with coefficients in $R$. If $A\subseteq X$ and $u \in\wh(X,R)$  we denote  by $u|_A=\iota^*(u)$,
the restriction of $u$ to $A$. Clearly, $u|_A=0$, whenever $A$ is a categorical subset.
The concept of category weight is motivated by the fact that certain non-zero
cohomology classes are `heavier' and are trivial when restricted to unions of two or more
categorical subsets.

\begin{definition}\label{defi:cwgt}
The \emph{category weight} of a non-zero cohomology class $u \in  H^q(X;R) $, $q\geq 1$,  denoted
$\cwgt (u)$,  is the maximal integer $k$, such that $u|_A=0$ for every closed subset $A\subseteq X$
with $\cat(A)\le k$.
\end{definition}
The category weight of a zero class is usually left undefined (although some authors define its weight to
be infinite).
The main properties of category weight are the following \cite{FadellHusseini1992}:
\begin{itemize}\label{prop:cwgt}
\item[(1)]{ $\cat(X) >  \cwgt (u) \geq 1$ for every non-zero $u\in\wh(X;R)$.}
\item[(2)] If the product of classes $u_1,\ldots,u_n\in \wh(X;R)$ is non-zero then $\cwgt$ is superadditive:
$$\cwgt( u_1 \cdots  u_n) \geq \sum_{i=1}^n \, \cwgt(u_i)$$
\item[(3)] Let  $p$ be a prime, and $\beta\colon H^q(X; \ZZ_p) \to H^{q+1}(X;\ZZ_p)$ the (mod $p$)-Bockstein
homomorphism. If $u \in  H^1(X;\ZZ_p)$ and $\beta(u) \neq 0$, then $\cwgt (\beta u) = 2$.
\end{itemize}
Unfortunately, $\cwgt(u) $ is hard to compute because it is not a homotopy invariant, i.e. there are
examples of homotopy equivalences $f\colon X \to Y$  and  elements $u \in  H^*(Y;R )$ with
$\cwgt(f^*(u)) \neq  \cwgt (u)$. This is why Rudyak \cite{Rudyak} (and independently Strom \cite{Strom})
introduced a homotopy invariant version of category weight and called it strict category weight.
Recall that the \emph{category of a map} $f\colon A\to X$ is the minimal
integer $k$ for which $A$ can be covered by $k$ open sets  $U_1,\ldots U_k$, such that
$f|_{U_i}$ is null-homotopic for $i=1,\ldots,k$.

\begin{definition}[\cite{Rudyak}]\label{defi:scwgt}
The strict category weight of a non-zero cohomology class $u\in\wh(X;R)$, denoted $\swgt (u)$,
is the maximal integer $k$, such that $f^*(u)=0$ for every map $f\colon A\to X$ with $\cat(f)\le k$.
Equivalently $\swgt (u)= \min \{\cwgt (h^*(u))\mid h\colon A\to X\ \text{a homotopy equivalence} \}$.
\end{definition}
Basic properties of strict category weight are listed below (cf. \cite{Rudyak}).
\begin{itemize}\label{prop:scwgt}
\item[(1)] $\cwgt(X)\ge \swgt (u) \geq 1$ for every non-zero $u\in\wh(X;R)$
\item[(2)]{If $f^*(u)\ne 0$ for some $f\colon Y \to  X$ and $u\in\wh(X;R)$, then $\swgt(f^*(u))\ge\swgt(u)$.\\
In particular, if $f$ is a homotopy equivalence, then $\swgt(f^*(u))=\swgt(u)$.}
\item[(3)]{If the product of classes $u_1,\ldots,u_n\in \wh(X;R)$ is non-zero then $\swgt$ is superadditive:
$$\swgt( u_1 \cdots  u_n) \geq \sum_{i=1}^n \, \swgt(u_i)$$}
\item[(4)]{$\swgt (u)\le |u|$, where $|u|$ denotes the dimension of a non-zero class $u$.}
\item[(5)] If $G$ is any discrete group, then  $\swgt(u)=|u|$ for all non-zero classes $u\in\wh(K(G,1);R)$.
\end{itemize}
In particular, if $f$ is a map from $X$ to an Eilenberg-MacLane space $K(G,1)$, then by (2),(4) and (5) we
conclude that $\swgt(f^*(u))=|u|$ whenever $f^*(u)\ne 0$. Note also that by (1),
the strict category weight provides a lower bound for the category weight of a cohomology class.
This is important, because category weight does not behave nicely with respect to pull-backs.
A case of interest for us arises in the cohomology of products: if $p\colon X\times Y\to X$ is the projection,
then for every non-zero element $u\in\wh(X;R)$ there is a non-zero element
$u\otimes 1=p^*(u)\in \wh(X\times Y;R)$. Then $\cwgt(u\otimes 1)$ can be smaller than $\cwgt(u)$ and
we may only conclude that  $\cwgt(u\otimes 1)\ge\swgt(u\otimes 1)\ge\swgt(u)$.

Below we will use category weight of classes in $\wh(X;R)$ to estimate strict covering type of $X$ and
strict category weight of classes in $\wh(X;R)$ to estimate covering type of $X$. This mismatch
is unfortunate and may be a bit confusing, nevertheless we decided to keep the standard terminology.

\section{Category weight estimates}\label{sec:Category weight estimates}

The main objective of this section is to derive estimates for the covering type of $X$ based on
the information on non-trivial products in cohomology and the (strict) category weight of the factors.

Ideally, we would like to have category weights to appear in our formula, but that is
impractical, because the precise values of $\cwgt$ or $\swgt$ are often unknown and we have
only certain estimates.
We thus take a more formal approach and assign to each homogeneous element $u\in\wh(X;R)$ an integer, called   
"weight estimator" ${\bf w}(u)$, satisfying $1\le {\bf w}(u)\le\swgt(u)$ (respectively
$1\le {\bf w}(u)\le\cwgt(u)$)
and view ${\bf w}(u)$ as the best available estimate for $\swgt(u)$ (resp. $\cwgt(u)$). We then extend the
definition of  $\bf w$
to finite sequences $s=(u_1,\ldots,u_n)$ consisting of homogeneous elements $u_i\in \wh(X;R)$
and satisfying the condition that the product $\widehat s:=u_1\cdots u_n\ne 0$.

\begin{definition}[Weight estimator]\label{def:weight estimator}
A {\bf{weight estimator}} is a function ${\bf w}$ that to every non-zero homogeneous class
$u\in \tilde{H}^*(X;R)$ assigns a positive integer ${\bf w}(u)$ satisfying the condition
${\bf w}(u) \le \cwgt(u)$. A weight estimator is {\bf strict} if ${\bf w}(u) \le \swgt(u)$.
Furthermore, for every sequence $s=(u_1,\ldots,u_n)$ of homogeneous cohomology classes,
such that  $\hat{s}:= u_1\, \cdots\ u_n \ne 0$ we let
$$ {\bf w}(u_1,\ldots,u_n):= {\bf w}(u_1) +\,\dots\, + {\bf w}(u_n).$$
For technical reasons we define the weight and dimension of an empty sequence $s$ as
${\bf w}(s):=0$ and $|\widehat{s}|:=-1$.
\end{definition}

Observe that for every sequence $s$ and every weight estimator we have
$\cwgt(\widehat s )\ge {\bf w}(s)$, and respectively  $\swgt(\widehat s )\ge {\bf w}(s)$,
if the estimator is strict.

Two sequences  are said to be \emph{equivalent}, $(u_1,\ldots,u_n)\sim (v_1,\ldots,v_n)$, if one of
them can be obtained
by permuting the elements of the other. Obviously, $s\sim s'$ implies ${\bf w}(s)={\bf w}(s')$.
Sequences $s'=(v_1,\ldots,v_k)$ and $s''=(v_{k+1},\ldots.v_n)$ are called
\emph{complementary subsequences} of $s=(u_1,\ldots,u_n)$ if $(v_1,\ldots,v_n)\sim (u_1,\ldots,u_n)$.
We will write $s'\le s$ whenever $s'$ is a subsequence of $s$ (or of a sequence equivalent to $s$).

We are going to use repeatedly the following result:

\begin{lemma}
\label{lem:nontrivial restriction}
Let $s$ be a sequence of homogeneous elements in $\wh(X;R)$, such that $\widehat s\ne 0$, and let
$s'\le s$. Assume that $X$ can be written as a union of open subsets
$X=U\cup V_1\cup\cdots\cup V_k$, where $V_1,\ldots,V_k$ are contractible. Then
${\bf w}(s')\le {\bf w}(s)-k$ implies $\widehat{s'}|_U\ne 0$ for every
weight estimator ${\bf w}$.
\end{lemma}
\begin{proof}
The result follows easily from the properties of the cup-product and the category weights. In the case of 
$\swgt$ if
$s''$ is a complementary sequence of $s'$ in $s$, then
$\swgt(\widehat{s''})\ge {\bf w}(s'')\ge k$. It follows that $\widehat{s''}|_{V_1\cup\cdots\cup V_k}=0$, as $\swgt(\widehat{s''})\le \cwgt(\widehat{s''})$. For $\cwgt$ we do not need the last
inequality, since it follows directly from the definition of $\cwgt$.
Therefore there exists a cohomology class $y\in \wh(X,V_1\cup\cdots\cup V_k;R)$, such that
$\widehat{s''}=j^*(y)$ for $j^*\colon \wh(X,V_1\cup\cdots\cup V_k;R)\to \wh(X;R)$. Analogously, if
$\widehat{s'}|_U\ne 0$, then there exists a cohomology class $x\in \wh(X,U;R)$, such that
$\widehat{s'}=j^*(x)$ for $j^*\colon \wh(X,U;R)\to \wh(X;R)$. By naturality of the cup-product, we obtain
$$\widehat{s}=\widehat{s'}\cdot\widehat{s''} =j^*(x)\cdot j^*(y)=j^*(x\cdot y)=0$$
because $x\cdot y\in \wh(X,U\cup V_1\cup\cdots\cup V_k;R)=\wh(X,X;R)=0$. This contradicts our assumption
that $\widehat{s}\ne 0$, therefore we must have $\widehat{s'}|_U\ne 0$.
\end{proof}

We are now ready to formulate and prove our main geometric theorem.

Given any weight estimator ${\bf w}$ we define $\wct(s;{\bf w})$  as
 \begin{equation}\label{wct(s;w)}
\wct(s;{\bf w}):= 1+{\bf w}(s)+\sum_{k=1}^{{\bf w}(s)} \max\bigg\{|\widehat{s'}|\ \bigg|\ s'\le s,\ {\bf w}(s')\le
k \bigg\}\,.
\end{equation}
We have the following theorem
\begin{theorem}\label{thm:weighted estimate}
Let $X$ be a space, $R$ a coefficient ring, and ${\bf w}$ a weight estimator.

Then for every sequence
$s$ of homogeneous elements of $\wh(X;R)$ for which $\widehat{s}\ne 0$ we have
$$\sct(X)\ge \wct(s;{\bf w}),$$

\end{theorem}
\begin{proof}
Let $\mathcal U$ be a good cover of $X$, and let $s=(u_1,\ldots,u_n)$ be a sequence of homogeneous elements
of $\wh(X;R)$, such that $\widehat{s}\ne 0$. For every $k=1,\ldots,{\bf w}(s)$ let $S_k$ be the
set of subsequences $s'$ of $s$, such that ${\bf w}(s')\le {\bf w}(s)-k+1$, and let
$d_k:=\max\big\{|\widehat{s'}|\ \big|\ s'\in S_k\big\}$, in particular $d_1=|u_1\cdots u_n|$.

By Nerve theorem,  there exists a subset $\VV_1\subseteq\mathcal U$ consisting of $d_1+1$ open sets that
have a non-empty intersection, otherwise the dimension of the nerve of $\mathcal U$ would be less than
$d_1$, contradicting non-triviality of $\widehat{s}$.
Let us denote by $\bigcup\VV_1$ the union of elements of $\VV_1$. By Nerve theorem $\bigcup\VV_1$ is
contractible,
thus Lemma \ref{lem:nontrivial restriction} implies that  $s'|_{\bigcup(\UU-\VV_1)}\ne 0$ for every
$s'\in S_1$. Repeating
the argument with $\UU-\VV_1$ in place of $\UU$ we find a subset $\VV_2\subseteq \UU-\VV_1$ with
at least $d_2+1$ elements, and such that $\bigcup\VV_2$ is contractible. Lemma \ref{lem:nontrivial restriction}
implies that
$s'|_{\bigcup(\UU-\VV_1-\VV_2)}\ne 0$ for every $s'\in S_2$. Continuing this procedure we obtain disjoint
sets $\VV_k\subseteq\UU$, $k=1,\ldots,{\bf w}(s)$. Note that some of the sets
$S_k$ may be empty, if ${\bf w}(u_i)>1$ for all $i$. Since the corresponding $d_k=-1$ by our definition
above, those
will not affect our summation of values $d_k+1$. Note also, that after removing the last non-empty
subset $\VV_k\subseteq \UU$, there must remain at least one contractible set in $\UU$, because
the contractible set $\bigcup\VV_k$ cannot support a non-zero cohomology class.
By summing up all the contributions we obtain an estimate for the number of elements in $\UU$ and
thus for the covering type of $X$:
$$\sct(X)\ge 1+\sum_{k=1}^{{\bf w}(s)} (1+\max\big\{|\widehat{s'}|\ \big|\ s'\in S_k\big\})=
1+{\bf w}(s)+\sum_{k=1}^{{\bf w}(s)} \max\big\{|\widehat{s'}|\ \big|\ {\bf w}(s')\le k\big\}.$$
Note that we have reversed the summation order in the second formula.
\end{proof}

The estimate that we have just derived depends on a number of choices, namely on weights estimators assigned
to elements of the cohomology ring and on a sequence of elements whose product is non-zero.
We may derive a homotopy invariant version of the estimate by defining
${\bf w}(u):=\swgt(u)$ and extending it to sequences as in  Definition \ref{def:weight estimator}, i.e.  
$\swgt(u_1,\ldots, u_n):=
\swgt(u_1)+\ldots+\swgt(u_n)$ (so that $\swgt(s)\le\swgt(\widehat s)$ whenever $\widehat s\ne 0$).
Then we define
$\swct(X)$ to be the maximum of values
\begin{equation}\label{swgt}
1+\swgt(s)+\sum_{k=1}^{\swgt(s)} \max\bigg\{|\widehat{s'}|\ \bigg|\ s'\le s,\ \swgt(s')\le
k \bigg\}
\end{equation}
over all coefficient rings $R$ and over all sequences $s$ of homogeneous elements of $\wh(X;R)$,
such that $\widehat{s}\ne 0$.

As the strict category weights of cohomology classes are invariant with respect to homotopy 
equivalences, we obtain the following consequence of Theorem \ref{thm:weighted estimate}:

\begin{theorem}\label{thm:swct estimate}
The integer $\swct(X)$ is a homotopy invariant of $X$ and is a lower bound for the covering type
of $X$, i.e.
$$\ct(X)\ge \swct(X)\,.$$
Note that for every strict weight estimator ${\bf w}$ and every sequence $s$  we have
$ \swct(X)\,\ge \,\wct(s; {\bf w})\, .$
\end{theorem}
\begin{proof}
Let $f\colon Y\to X$ be a homotopy equivalence, and let $\UU$ be a good cover of $Y$. Since the strict
category weight is invariant with respect to homotopy equivalence, we have for every sequence
$s=(u_1,\ldots,u_n)$ of elements in $\wh(X;R)$ that $\swgt(u_1,\ldots,u_n)=\swgt(f^*(u_1),
\ldots,f^*(u_n))$,
hence by Theorem \ref{thm:weighted estimate} the cardinality of $\UU$ is an upper bound for
$\wct(f^*(s);\swgt)$.

We conclude that $\ct(X)$, the minimum of cardinalities of good covers
$\UU$, exceeds $\swct(X)$,  the maximum of weighted estimates over all sequences of cohomology classes
with non-zero product. It shows the first part of theorem.
\end{proof}

As a direct consequence of (\ref{Delta}) and Theorems \ref{thm:weighted estimate} and \ref{thm:swct estimate} 
we get the following.
\begin{corollary}\label{cor:Delta by wct}
Let $X$ be a triangulable space. For every weight estimator ${\bf w}$ and every every sequence
$s$ of  homogenous elements of $\wh(X;R)$ with $\widehat{s} \ne 0$    we have
$$ \Delta(X)\ge \wct(s; {\bf w}).$$
\end{corollary}

We will see later that it is occasionally possible to determine products of cohomology classes that
maximize the above expression and thus to compute the exact value of $\swct(X).$

\begin{remark}\label{rem:comparision with old cohomological ct}
The function $\swct$ is a weighted variant of $\ct$ of a product of cohomology classes, that we introduced in
\cite[Section 3]{GovcMarzantowiczPavesic2019}. In fact, $\swct$ reduces to the weighted variant of  $\ct$ if 
we set ${\bf w}(u_i):=1$ for all non-trivial elements of $\wh(X;R)$.

To prove the claim we order the factors of a non-trivial product $u=u_1\cdots u_n$ so that their dimensions
are increasing $\vert  u_1 \vert \leq \vert u_2  \vert \leq \, \cdots\, \leq  \vert  u_n \vert $.
In our case the weight of every subsequence is equal to its length so the formula from Theorem
\ref{thm:weighted estimate} assumes the form
$$ 1+n+\sum_{k=1}^n \max\big\{|s|\ \big|\ s\ \text{is a subproduct of u of length at most k}\big\}.
$$
Since the dimensions of factors of $u$ increase, the maximal dimension is achieved by the product
$u_{n-k+1}\cdots u_n$.
Summing up all contributions, we get the value
$ n+1 + \sum _{k=1}^n \,k\, \vert u_k\vert $
which is precisely the formula of \cite[Section 3]{GovcMarzantowiczPavesic2019}.

Observe that in the non-weighted case it is usually fairly easy to determine
products that give best lower bounds for the covering type of $X$.
The weighted variant requires more care as we must determine among sub-products
of certain weight those that have maximal dimension.
\end{remark}

\begin{remark}
\label{rem:improvements}
The estimate $\sct(X)\ge \wct(s;{\bf w})$ takes into account non-trivial products of cohomology
classes. Occasionally the product that gives the best estimate is not a top-dimensional class, and then
Theorem \ref{thm:weighted estimate}  can be improved by taking into
account $\hdim(X)$, the homotopy dimension of $X$. In fact, in the
first step of the proof of Theorem \ref{thm:weighted estimate} the Nerve theorem implies that one can find
a subset $\VV_1\subseteq \UU$ with $\hdim(X)+1$ elements that have non-empty intersection
(instead of $|\widehat{s}|+1$).  The rest of the argument remains unchanged, and by summing up all
contributions we obtain a better (but less elegant) estimate
$$\sct(X)\ge \wct(s;{\bf w})+\hdim(X)- |\widehat{s}| .$$
A similar improvement is valid for Theorem \ref{thm:swct estimate}, as well.

Another slight improvement is available, if there are at least two linearly independent factors of same
minimal weight in the product sequence $s=(u_1,\ldots,u_n)$. If that happens, we may improve
the estimate by one - compare \cite[Lemma 3.4 and Theorem 3.5]{GovcMarzantowiczPavesic2019}.
We leave the details to the reader.
\end{remark}

There are two main reasons why the formal description of weighted estimates of covering type is quite
complicated, both technically and notionally. First, the exact category weights of cohomology
classes are in general unknown so we must work with estimates. As a consequence, it is practically
impossible to give a coherent and at the same time computable definition of weights for products of
various classes. We are thus forced to work with sequences of elements instead of their products.
However, in practice most of these difficulties can be avoided, as we usually consider a particular
non-trivial product of cohomology classes, so it is sufficient to specify only the weights of its
factors and examine only the subproducts of that product. In the following examples we will assign
weights only to certain generators of $\wh(X;R)$ and work with products and their sub-products
instead of sequences of elements and their subsequences.

Since the computational algorithm is quite involved, we begin with a somewhat artificial
example that explains how to determine $\wct$ of a specific cohomological product.

\begin{example}\label{ex:absract wct}
We compute $\wct(x^2yz;{\bf w})$ where ${\bf w}(x)=|x|=2, {\bf w}(y)=2, |y|=3$ and ${\bf w}(z)=3, |z|=5$.
Since ${\bf w}(x^2yz)=9$ we must determine sub-products of $x^2yz$ whose total weight is at most $k$, for
$k=1,2,\ldots,9$, and then for each $k$ find those with highest dimension. These are given in the
following table, with highest dimensional products written in boldface:

\begin{center}
\begin{tabular}{l|l|c}
$k$ & subproducts $s$ of $x^2yz$ with ${\bf w}(s)\le k$ & dimension\\ \hline
$1$ & -- & -1\\
$2$ & $x,\pmb{y}$ & 3\\
$3$ & $x,y,\pmb{z}$ & 5\\
$4$ & $x,y,\pmb{z},x^2,\pmb{xy}$ & 5\\
$5$ & $x,y,z,x^2,xy,xz,\pmb{yz}$ & 8\\
$6$ & $x,y,z,x^2,xy,xz,\pmb{yz},x^2y$ & 8\\
$7$ & $x,y,z,x^2,xy,xz,yz,x^2y,x^2z,\pmb{xyz}$ & 10\\
$8$ & $x,y,z,x^2,xy,xz,yz,x^2y,x^2z,\pmb{xyz}$ & 10\\
$9$ & $x,y,z,x^2,xy,xz,yz,x^2y,x^2z,xyz,\pmb{x^2yz}$ & 12\\
\end{tabular}
\end{center}
Thus we have $\wct(x^2yz;{\bf w})=1+9+(-1+3+5+5+8+8+10+10+12)=70$.
\end{example}

In our second example we estimate the covering type and the size of a minimal triangulation of
a lens space. To do it, we employ ${\bf w}(u) =  \swgt(u)$ for indecomposable homogenous element $u$.

\begin{example}\label{ex:lens space}
Let $p$ be a prime and let $L=L^{2n+1}(p)$ be the lens space obtained as the quotient of $S^{2n+1}$
by the standard action of the cyclic group $\ZZ_p$. It is well-known that $H^{2n+1}(L^{2n+1}(p);\ZZ_p)$,
the top-dimensional cohomology
group, is generated by the product $xy^n$. If ${\bf w}(u)= \swgt(u)$ then ${\bf w}(x)=|x|=1$ and
${\bf w}(y)=|y|=2$  as $\cwgt(y)=\swgt(y)=2$, because
$y=\beta(x)$.
Clearly, for every $k\le 2n+1$ the sub-product of $xy^n$ whose weight ${\bf w}$ is at most $k$
and whose dimension is maximal is
$xy^i$ for $k=2i+1$ and $y^i$ for $k=2i$.
Thus we have
$$\swct(L^{2n+1}(p))\ge \wct(xy^n;{\bf w})=1+{\bf w}(xy^n)+(1+2+\ldots+ 2n+1)=(n+1)(2n+3)
$$
and by Theorem \ref{thm:swct estimate}
$$\Delta(L^{2n+1}(p))\ge\ct(L^{2n+1}(p))\ge 1+(2n+1)+n(2n+1)=(n+1)(2n+3).$$

As a comparison, the non-weighted cohomology estimate in
\cite[Prop. 3.1 and Theorem 3.5]{GovcMarzantowiczPavesic2019},  yields
$$\ct(L^{2n+1}(p))\ge 1+2(1+\ldots+n)+(n+2)=(n+1)^2+2,$$
which is, of course, considerably smaller.

We should mention that \cite[Theorem 2.2]{GovcMarzantowiczPavesic2019} gives a lower bound for
the covering type in terms of the Lusternik-Schnirelmann category:
$$\ct(L^{2n+1}(p))\ge \frac{1}{2}\cat(L^{2n+1}(p))\big(\cat(L^{2n+1}(p))+1\big)=(n+1)(2n+3)$$
since $\cat(L^{2n+1}(p))=2n+2$. The reason is that the category of lens spaces is the maximal possible
with respect to their dimension.
\end{example}

However, the power of our new method becomes evident when we try to estimate the covering type
of a product of a lens space and a sphere. The result will be used on several occasions in the
rest of the paper so we state it as a proposition.

\begin{proposition} \label{prop:lens space x sphere}
$$\ct(L^{2n+1}(p)\times S^m)\ge (n+1)(2m+2n+3)+2 \,.$$
\end{proposition}
\begin{proof}
Like in the previous example we consider the $\ZZ_p$-cohomology ring and
find a non-zero product $xy^nz\in H^*(L^{2n+1}(p)\times S^m;\ZZ_p)$ where $|x|=1,|y|=2$, $|z|=m$
and $\swgt(y)=2$.
Then ${\bf w}(xy^nz)=2n+2$ and products of maximal dimension
 in increasing order of weight are
$$z,xz,yz,xyz,y^2z,xy^2z,\ldots,xy^nz.$$
Thus we can apply   Theorem \ref{thm:swct estimate}, together with Remark \ref{rem:improvements}
(since there are two elements, $x$ and $z$, with minimal weight) to  obtain
$$\ct(L^{2n+1}(p)\times S^m)\ge \wct(xy^nz;{\bf w})+1=$$
$$=1+(2n+2)+\big(m+(m+1)+\ldots+(m+2n+1)\big)=(n+1)(2m+2n+3)+2 \,.$$
\end{proof}

For a comparison, the improved category estimate of \cite[Theorem 2.3]{GovcMarzantowiczPavesic2019} gives
$$\ct(L^{2n+1}(p)\times S^m)\ge (n+2)(2n+3)+m-2,$$
while the non-weighted cohomology estimate of \cite[Theorem 3.5]{GovcMarzantowiczPavesic2019} gives
$$\ct(L^{2n+1}(p)\times S^m)\ge (n+2)(m+n+1)+n+3.$$

Our final  simple example presents a lower bound for the strict covering type of the symplectic group
$Sp(2)$. In this example we use as the weight ${\bf w}$ of an indecomposable homogenous  element the
category weight $\cwgt$.

\begin{example} \label{ex:Sp(2)}
The integral cohomology of the symplectic Lie group $Sp(2)$ is given as
$\wh(Sp(2))\cong \bigwedge(x,y)$ with $|x|=3$ and $|y|=7$. Its covering type has been estimated in
\cite[Prop. 2.10]{DuanMarzantowiczZhao2021} using the non-weighted cohomology estimate of
\cite[Theorem 3.5]{GovcMarzantowiczPavesic2019}, which yields $\Delta(Sp(2))\ge\ct(Sp(2))\ge 21$.

However, Fadell and Husseini \cite[Remark 3.14]{FadellHusseini1992} showed that the category weight
of the second generator is $\cwgt(y)=2$.  With  ${\bf w} (u)= \cwgt(u)$, this  gives an improved estimate,
based on Theorem \ref{thm:weighted estimate}:
$$\Delta(Sp(2))\ge  \sct(Sp(2))\ge\wct(xy)=1+{\bf w}(xy)+(|x|+|y|+|xy|)=24 \,.$$
\end{example}

\section{Covering type of spaces with finite cyclic fundamental group}
\label{sec:Manifolds with finite cyclic fundamental group}

If $X$ is a simply-connected space upon which a finite group $G$ acts, then $G$ is isomorphic
to the fundamental group of its orbit space, $\pi_1(X/G)\cong G$, and the quotient map
$X\to X/G$ is a universal covering projection. Alternatively we may start from a space
with a finite fundamental group $G$ and view it as a quotient of its universal covering
with respect to the free action of $G$. In this section we will constantly alternate between
these two equivalent viewpoints.

The cohomology rings of $X$ and of $X/G$ are related by a spectral sequence that we describe below.
Most importantly for our purposes, the action of $G$
gives rise to classes of higher weight in the cohomology of $X/G$, which are crucial for our
estimates of $\ct(X/G)$. We are going to systematically derive lower bounds for the covering type
and the size of triangulations of spaces that arise as quotients of free cyclic actions on spaces that
are cohomologically like spheres or products of spheres.

To obtain cohomological information about quotients of  highly-connected, or more generally,
highly-acyclic spaces we
study the structure of skeleta of a classifying space $BG$ of a finite group $G$. This will give a
weighted  estimate of the covering type $\ct(X/G)$ of the  orbit space of a free action of $G$ on an
$n$-acyclic space $X$. The main example is $G=\ZZ_m$  where the calculations are similar to those of
Example \ref{ex:lens space} and Proposition \ref{prop:lens space x sphere}.

Let $EG$ be the universal space of $G$, i.e. a contractible  CW-complex with a free action of $G$, and let
$ BG=EG/G$ be the classifying space of $G$.  There are many models of $EG$, e.g. Milnor's infinite 
join
$EG:= G * G * G *\cdots $. Note that $BG $ is an Eilenberg-MacLane space of the type $K(G,1)$, i.e.
$\pi_1(BG) = G$ and  $ \pi_i(BG)=0$ for $i> 1$.  
The cohomology  $H^*(BG;R)= H^*(K(G,1);R)$ is often called the cohomology of $G$ and is 
denoted as $H^*(G; R)$.

Recall that the strict category weight of every non-zero class $u\in H^*(BG;R)$ satisfies 
$\swgt(u)=|u|$. Thus, if we let ${\bf w}(u):=\swgt(u)$, then 
${\bf w}(s)=|\hat s|=\swgt(\hat s)$ for every sequence $s$ in $H^*(BG;R)$ with $\hat s\ne 0$.
This leads to a simple but useful estimate of weighted covering type of a skeleton of $BG$. 

\begin{lemma}\label{lem:estimate of skeleton}
For every $n>0$ we have 
$$\swct(BG^{(n)})\ge \max\big\{ \wct(s;{\bf \swgt})\mid s 
\text{ a sequence of classes in } H^*(BG;R), |\hat s|\le n\big\}.$$
Equality holds if all elements in $H^n(BG;R)$ are decomposable. 
\end{lemma}
\begin{proof}
The inclusion $i:BG^{(n)}\hookrightarrow BG$ induces homomorphisms 
$i^*\colon H^j(BG;R)\to H^j(BG^{(n)};R)$
which are bijective if $j<n$ and injective if $j=n$. Therefore, every sequence 
$u_1,\ldots,u_k$ in $H^*(BG;R)$ with $u_1\cdots u_k\ne 0$ and $|u_1\cdots u_k|\le n$ 
yields a sequence $i^*(u_1),\ldots, i^*(u_k)$ in $H^*(BG^{(n)};R)$ with  
$i^*(u_1)\cdots  i^*(u_k)=i^*(u_1\cdots u_k)\ne 0$ and 
${\bf w}(i^*(u_1),\ldots, i^*(u_k))={\bf w}(u_1,\ldots,u_k)$. This immediately implies the stated
inequality. If all elements of $H^n(BG^{(n)};R)$ are decomposable, then 
$H^n(BG;R)\cong H^n(BG^{(n)};R)$ and the last claim follows.
\end{proof}

Before proceeding, we have to recall few more facts of the theory of transformation groups.
Let $X$ be a $G$-space,  $EG$ be the universal space of the group $G$, and  $BG=EG/G$ be the classifying space 
of $G$. The product
$X\times EG $ is a $G$-space with the action on coordinates  and its orbit space  is denoted by 
$X\times_G EG$. Note that the projections $p_1\colon X \times  EG \to X$, and correspondingly $p_2: X
\times  EG \to EG$ are $G$-equivariant. The induced maps of orbit spaces, denoted by the same letters, form 
projection maps of the corresponding fibrations  $p_2: X/G\subset X\times_G EG \to BG$
and $p_1: X\times_G EG \to X/G$ respectively. By the definition, (cf. \cite{Hsiang}) the Borel  cohomology of 
a $G$-space $X$ is  defined as

\centerline{$H_G(X;R):= H^*(X\times_G EG;R) = {\underset{d\to \infty}\lim\,} H^*(X\times_G EG^{(d)}; R)$,
where  $EG^{(d)}$  is the $d$-skeleton of $EG$.}
Furthermore, the fiber of $p_2$ is equal to $X$, and the fiber of $p_1$ over $[x]\in X/G$ is equal to $EG/
G_x$, where $G_x$ is the isotropy group of $x$.
Consequently, if $X$ is a free $G$-space then every fiber is equal to $EG$, thus  contractible. This leads
to the well known Borel isomorphism: if $X$ is a free $G$-space then
$$ H_G^*(X;R) = H^*(X/G; R)\,, $$ by  the Vietoris-Beagle theorem applied to the map $p_1: X\times_G EG
\to X/G$.
The second fibration $p_2: X \subset  X \times_G EG \to BG$ is locally trivial thus a Serre fibration,
called the Borel fibration.   It is used to study the cohomology of $X/G$ (cf. \cite{Hsiang}). We will
denote its projection $p_2$ shortly  by  $p$.
If $X$ is a free numerable  $G$-space then there exists a classifying map $f: X/G \to BG$ such that the
principal $G$-bundle $X {\overset{\pi} \to}  X/G$ is induced by $f$, i.e. it is isomorphic to the induced
bundle $f^*(EG, \pi, BG)$ and $H(f)^*$ equips $H^*(X/G)$ in the structure of $H^*(BG)$-module.
 To simplify algebraic consideration, from now on we assume that the coefficient ring $R$ is a field $F$.

\subsection{$n$-acyclic spaces}\label{n+1-acyclic space}
We are going to prove  that if $X$ is $n$-connected
(or more generally $n$-acyclic),
then  there are products  $s$ in $H^*(X/G;F)$  of same weight as the corresponding elements in
$H^*(BG;F)$, provided that the dimension of the resulting product is at most $n$. 
A proof of this fact is by routine spectral sequence computation but we include it for 
the  convenience of the reader.
Recall that a space $X$ is called  $n$-acyclic over a ring $R$ if the homology groups vanish, i.e. $
\tilde{H}_i(X;R)= 0$ for $i\leq n$. If $R=\ZZ$ we call this space  shortly $n$-acyclic. Clearly an
$n$-connected space $X$  is $n$-acyclic. A space $X$ is called cohomology  $n$-acyclic over a ring 
$R$ if the
cohomology groups vanish, i.e. $ \tilde{H}_i(X;R)= 0$ for $i\leq n$. 
By the universal coefficient theorem
if $X$ is $n$-acyclic then it is cohomology $n$-acyclic for every $\ZZ$-module, e.g. any field $F$.

\begin{theorem}\label{thm:acyclic} Let $X$  be a finite-dimensional compact space on which a finite 
group $G$ acts freely and assume that $X$ is $n$-acyclic.\\[2mm]
(a) The classifying map  $f\colon X/G\to BG$ induces homomorphisms 
$f^*\colon H^j(BG;F)\to H^j(X/G;F)$ which are bijective if $j\le n$ and injective if $j=n+1$.\\
(b) If all elements of $H^{n+1}(BG^{(n+1)};F) $ are decomposable, then 
$\swct(X/G) \geq \swct (BG^{(n+1)})$. In particular, 
$\ct (X/G) \geq \swct(BG^{(n+1)})$ and also $\Delta(X/G) \geq \swct(BG^{(n+1)})$ 
provided $X$ is triangulable.
\end{theorem}

\begin{proof}
Since $p\colon X \times_G EG \to BG$ is a Serre fibration with fibre $X$ we can use the cohomology 
Leray-Serre spectral sequence $(E^{p,q}_r, d_r)$ of it, which converges to $H^*( X\times_G EG; R)$. As
a matter  of  fact, since $X$ is finite dimensional, this spectral sequence stabilizes, i.e. 
$d_r=0$ if $r\geq r_0$. By the definition of the spectral sequence 
$E_2^{p,q} = H^p(BG; {\mathcal{H}}^q(X;R))$, where ${\mathcal H}^q(X;R)$ is the local coefficients
system given by the action of $\pi_1(BG)= G$ on $H^*(X;R)$.

Now, assume that $\tilde{H}^i(X;R)=0$ for $1\leq i \leq n$. Then the table of $E^{p,q}_2$ has the zero level 
row equal to   $ E^{*,0}_2 =H^*(BG;R)$, and next $n$ rows $E^{*,j}_2=
H^*(BG; {\mathcal H}^j(X;R))$  vanish for $1\leq j \leq n$. Clearly, 
$E_\infty^{*,j}=E^{*,j}_2$ and therefore $H^j(X \times_G EG;R)\cong H^j(BG;R)$ for $j\le n$. 
For $j=n+1$  we have two (potentially) non-trivial groups but we are only interested in the fact that 
$H^{(n+1)}(BG;R)$ injects in $H^{(n+1)}(X \times_G EG;R)$. In both cases the homomorphisms are 
induced by the projection $p\colon  X \times_G EG \to BG$. It is well-known that $p$ factors 
up to homotopy as $X \times_G EG \stackrel{p_1}{\longrightarrow} X \stackrel{f}{\longrightarrow} BG$
and that $p_1$ is a homotopy equivalence when the group action is free. Thus we may conclude that 
$f^*\colon H^j(BG;F)\to H^j(X/G;F)$ is bijective if $j\le n$ and injective if $j=n+1$.

Part (b) follows from Lemma \ref{lem:estimate of skeleton} and the fact that $f^*$ is an isomorphism
for $j\le n$.
\end{proof}

Let us illustrate Theorem \ref{thm:acyclic} in the case of $G=\ZZ_m$ the cyclic group of order $m$.
For $G=\ZZ_m$ one can take $EG=S^\infty = \lim_{k \to \infty}\, S^{2k+1}$, where $S^{2k+1}\subset \CC^{k+1}$ is the unit sphere in the complex Euclidean space  and $ \ZZ_m \subset  S^1$ is the
group  of $m$-th roots of unity acting as the scalars. It follows that  a model of $B\ZZ_m$ is the
infinite dimensional lens space $L^\infty(m) = S^\infty/ \ZZ_m$, where $S^\infty = {\underset{n\to \infty}\lim
\,} S^{2n+1}$ is a contractible  space with the free action of $\ZZ_m \subset S^1$  considered as the complex
numbers acting on the complex spheres.

We begin with recalling  the description of the cohomology ring of $B\ZZ_m$ (\cite[Chapt. III]{Hatcher2002})
\begin{equation}
\begin{matrix} H^*(B\ZZ_m; \ZZ_m) \simeq \ZZ_m[x,y]/ (x^2= k\, y)\,, \;\; {\text{where}}\;\;{ \begin{cases} k = 0 \;\;{\text{if}}\;\; m\;\;{\text{is odd}}\cr k=m/2 \;\;{\text{if}}\;\; m \;\;
{\text{is even}}\end{cases}} \cr
 \vert x\vert = 1,  \;\; \vert y \vert = 2, \;\; {\text{and} } \;\; y=\beta(x)\;\;{\text{where}} \;\; \beta\;\;{\text{is the Bockstein homomorphism\,.}}\end{matrix}
 \end{equation}

 As it is more convenient to work with a field as the coefficient ring we will use the following fact
(see \cite[3E,Exercise 1]{Hatcher2002})
 \begin{fact}\label{fact:Z_k cohomology of lens}
 Let $ k\mid m$  be a number dividing $m$.  Then $ H^*((\ZZ_m, 1); \ZZ_k)$ is isomorphic as a ring to $H^*(K(\ZZ_m, 1); \ZZ_m)\otimes \ZZ_k$. In particular, if $m/k$  is even, this is
 $  \bigwedge_{\ZZ_k}(x)\otimes \ZZ_k[y]$, where $ \bigwedge_{\ZZ_k}(x)$ is the skew-symmetric algebra over $\ZZ_k$.
 \end{fact}

\zz{  We summarize  it as follows:
\todo{TODO: IS THE NEXT STATEMENT SUPPOSED TO BE A PROPOSITION?}
\begin{abc}\label{cohomology of cyclic group}
 {\phantom{cohomology of cyclic group}}

 \begin{itemize}
 \item[1.] {If $m$ is odd then we can take as $p$ any prime divisor of $m$.  Then \\ $ \todo{H^*(\ZZ_m, 1); \ZZ_p)} \simeq \ZZ_p[x,y]/ (x^2= 0,\, y) \simeq   \bigwedge_{\ZZ_p}(x)\otimes \ZZ_p[y]$
     with $\vert y\vert =2$, $\vert x \vert =1$, and $y=\beta(x)$. }
 \item[2.] {If $m$ is even but not a power of $2$ then $m=m' \, 2^s$ ,  where $m>2$ is odd.  Then we take as $p$ any prime divisor of $m'$ and by the above we have \\
 $\todo{H^*(\ZZ_m, 1); \ZZ_p)} \simeq \bigwedge_{\ZZ_p}(x)\otimes \ZZ_p[y]$  $\vert y\vert =2$, $\vert x \vert =1$, and $y=\beta(x)$.}

 \item[3.] {If $m=2^s$, $ s>1$ is a power of $2$ greater than $2$ then we take  $p=2$. Once more by the above 
 we get  $\todo{H^*(\ZZ_m, 1); \ZZ_2)} \simeq \bigwedge_{\ZZ_2}(x)\otimes \ZZ_2[y]$
     with $\vert y\vert =2$, $\vert x \vert =1$, and $y=\beta(x)$.}

     \item[4.] {Finally if $m=2$ then   $\todo{H^*(\ZZ_2, 1); \ZZ_2)} \simeq \ZZ_2[x]$.}
\end{itemize}
\end{abc}}

As a direct consequence of this description we get the following.

\begin{proposition}\label{prop:acyclic for cyclic odd} Let $X$  be a finite-dimensional compact space 
on which the cyclic group  $G=\ZZ_m$, $m>2$, acts freely and assume that $X$ is $n$-acyclic. 
If $n$ is even, then $\swct(X/G) \geq  \frac{(n+2)(n+3)}{2}$, while if $n$ is odd, then
$\swct(X/G) \geq  \frac{(n+1)(n+2)}{2}$.

In particular, if $X$ is is a cohomology sphere over $\ZZ_m$,  then $n$ is even,  and if 
moreover the action preserves the orientation then 
$\swct(X/G)=\swct(L^{n+1}(\ZZ_m)) = \frac{(n+2)(n+3)}{2}.$ 

\end{proposition}
\begin{proof}

First note that to estimate $\swct(X/G)$ we can use the cohomology with coefficients in  the field $F_p$,  
with prime  $\,p \mid m\,$ described in
(\cite[Chapt. III]{Hatcher2002}, cf. Fact \ref{fact:Z_k cohomology of lens}). Next use Theorem 
\ref{thm:acyclic} which reduces the problem to estimating the value of $\swct ( L^{n}(\ZZ_m))$. But the 
argument  for the latter is the same as
that of Example  \ref{ex:lens space} as follows from the description of  $H^*( L^{n}(\ZZ_m);F_p)$ of 
(\cite[Chapt. III]{Hatcher2002}, cf. Fact \ref{fact:Z_k cohomology of lens}). Finally, we have to distinguish 
the case when $n$
is odd, and $n$ is even, because for the latter $B\ZZ_m^{(n)} = B\ZZ_m^{(n-1)} $ for the discussed model (it 
is also a consequence of the form of products of indecomposable homogenous  elements  in
$H^*( B\ZZ_m^{(n)} )$.

Note that if the action of $\ZZ_m$   is free then $\, m\mid \chi(S^{n+1})\,$ which implies $n+1$ is 
odd. The assumption $X\sim S^{n+1}$ yields that  the Borel  spectral sequence of $(X\times_G EG,
BG)$ is   algebraically the same as that of $S^{n+1}$ with the antipodal action, or $S^{2k+1} \subset 
\CC^{k+1}$  with the canonical linear action of $\ZZ_m$ respectively.  This leads to the second
part of the conclusion of statement.
\end{proof}

\begin{proposition}\label{prop:estimate n-acyclic m =2}
Let $X$  be a finite-dimensional compact space on which the group  $G=\ZZ_2$  acts freely. Assume that 
$X$ is $n+1$-acyclic.
Then $\ct (X/G) \geq \swct(X/G) \geq \swct(\RR P^n) = \frac{(n+1)(n+2)}{2}$. Moreover
if $ X \sim S^{n+1} $ is a $\ZZ_2$-cohomology sphere then $ \ct(X/G) \geq \swct(X/G) = 
\swct(\RR P^{n+1}) = \frac{(n+2)(n+3)}{2}$.
\end{proposition}

As a direct consequence of Propositions \ref{prop:acyclic for cyclic odd} and 
\ref{prop:estimate n-acyclic m =2} we get the estimate of  $\Delta(X/G) $ for spaces which are orbit spaces 
of free actions of the cyclic group $\ZZ_m$  on $n+1$-acyclic spaces.

\subsection{Products of spheres}\label{products of spheres}

We will next proceed to a systematic study of manifolds that arise as quotients of actions of
cyclic groups on manifolds whose $\ZZ_p$-cohomology rings are exterior algebras on odd generators.
It is remarkable that in most cases the algebraic computation of weighted estimates
will be very similar to those in Proposition \ref{prop:lens space x sphere}.

Let $M$ be a closed manifold with $\pi_1(M)\cong\ZZ_p$ where $p$ is an odd prime.
Then the universal covering $\widetilde M$ of $M$ is also a closed manifold and
$M=\widetilde M/\ZZ_p$. If the $\ZZ_p$-cohomology ring of $\widetilde M$ is isomorphic
to the $\ZZ_p$-cohomology ring of a product of two odd spheres (which we will denote as
$M\sim_p S^m\times S^n$, $m\le n$ odd), then the $\ZZ_p$-cohomology ring of $M$ was computed by
Dotzel, Singh and Tripathi \cite[Theorem 1]{DotzelSinghTripathi2001},
by applying the Leray-Serre spectral sequence to the Borel fibration
$$\widetilde M\hookrightarrow E\ZZ_p\times_{\ZZ_p} \widetilde M \to B\ZZ_p$$
(Note that for free actions the space $E\ZZ_p\times_{\ZZ_p} \widetilde M$ is homotopy equivalent
to $\widetilde M\big/\ZZ_p=M$.)

In fact, Theorem 1 in \cite{DotzelSinghTripathi2001} has three sub-cases and the formulation
therein does not explicitly state when each of them arise. However, this information can be
deduced from the proofs, so we state the relevant part for the convenience of the reader.

\begin{proposition} $($\cite[Theorem 1]{DotzelSinghTripathi2001}$)$
\label{prop: DotzelSinghTripathi}
Let $p, m,n$ be odd positive integers with $m\le n$ and let $\ZZ_p$ act freely on a manifold
$\widetilde M\sim_p S^m\times S^n$. Then the cohomology ring of the quotient manifold
$M:=\widetilde M/\ZZ_p$ is of one of the following
two types, depending on the value of the transgression homomorphism
$\tau_m\colon H^m(\widetilde M)\to H^{m+1}(B\ZZ_p)$ in the Borel fibration
$\widetilde M\hookrightarrow E\ZZ_p\times_{\ZZ_p} \widetilde M \to B\ZZ_p.$

If $\tau_m$ is trivial, then
$$H^*(M;\ZZ_p)\cong \ZZ_p[x,y,z]/(x^2,y^{\frac{m+1}{2}},z^2),$$
where $|x|=1$, $|y|=2$, $|z|=n$ and $y=\beta(x)$.

Otherwise, if $\tau_m$ is non-trivial, then
$$H^*(M;\ZZ_p)\cong \ZZ_p[x,y,z]/(x^2,y^{\frac{n+1}{2}},z^2),$$
where $|x|=1$, $|y|=2$, $|z|=m$ and $y=\beta(x)$.
\end{proposition}

Observe that under the assumptions of the theorem we have algebra isomorphisms
$H^*(M;\ZZ_p)\cong H^*(L_p^m\times S^n;\ZZ_p)$ if $\tau_m$ is trivial, and
$H^*(M;\ZZ_p)\cong H^*(L_p^n\times S^m;\ZZ_p)$ if $\tau_m$ is non-trivial.
Thus, we may use Theorem \ref{thm:swct estimate} and Corollary \ref{cor:Delta by wct} to obtain an estimate of
$\ct(M)$ and $\Delta(M)$ similarly as in Proposition \ref{prop:lens space x sphere}.

\begin{theorem}\label{thm:product of two spheres}
Let $M$ be a closed manifold with $\pi_1(M)=\ZZ_p$ ($p$ an odd prime) and whose universal covering
is $p$-equivalent to a product of two odd-dimensional spheres, $\widetilde M\sim_p S^m\times S^n$
for $m,n$ odd.
Then a lower bound for the covering type (and thus for the number of vertices in a triangulation)
of $M$ is
$$\Delta(M)\ge\ct(M)\ge \left\{\begin{array}{cc}
\frac{1}{2} (n+1)(2m+n+2)+2 ; \text{if } \tau_m=0\\[3mm]
\frac{1}{2} (m+1)(2n+m+2)+2 ; \text{if } \tau_m\ne 0
\end{array}\right.$$
where
$\tau_m\colon H^m(\widetilde M;\ZZ_p)\to H^{m+1}(B\ZZ_p;\ZZ_p)$ is the transgression
in the spectral sequence of the Borel fibration
$\widetilde M\hookrightarrow E\ZZ_p\times_{\ZZ_p} \widetilde M \to B\ZZ_p.$
\end{theorem}

Typical manifolds to which the above theorem applies are obtained by considering an $S^m$-bundle with
base $S^n$ ($m\le n$) and taking the orbit space with respect to a free $\ZZ_p$-action. This
includes quotients of free $\ZZ_p$-actions on products of spheres but also on many other spaces,
like for example Stiefel manifolds $V_2(\CC^n)$ and $V_2(\HH^n)$.

The above result can be extended to arbitrary free cyclic actions on spaces that are
$p$-equivalent to products of any number of spheres. Then a direct spectral sequence computation
of $H^*(M;\ZZ_p)$ becomes quite prohibitive but something
can still be said if we restrict the range of dimensions of the corresponding spheres.
We will express the result in terms of weighted covering type of a specific non-trivial product
as in Theorem \ref{thm:weighted estimate}.

\begin{theorem}\label{thm:product of spheres}
Let $M$ be a closed manifold with $\pi_1(M)=\ZZ_d$ and such that
 for some odd prime $p$ dividing $d$ the universal covering of $M$ is $p$-equivalent to a
product of odd-dimensional spheres ($\widetilde M\sim_p S^{m_1}\times \ldots\times S^{m_n}$ for
$m_1\le\ldots \le m_n$ odd). Assume in addition that $m_1+m_2> m_n$ and  that the induced action of $\ZZ_d$ on $H^*(\widetilde M;\ZZ_p)$
is trivial. Then exactly one among the transgression homomorphisms
$$\tau_{m_i}\colon  H^{m_i}(\widetilde M;\ZZ_p)\to H^{m_i+1}(B\ZZ_d;\ZZ_p)$$
associated to the Borel fibration $\widetilde M\hookrightarrow E\ZZ_d\times_{\ZZ_d} \widetilde M \to B\ZZ_d$
is non-trivial and the lower bound for the covering type and the number of vertices in a triangulation
of $M$ is given as ( $\widehat{\ }$ denotes omission)
$$\Delta(M)\ge\ct(M)\ge \wct(xy^{\frac{m_i-1}{2}}z_1\cdots\widehat{z_i}\cdots z_n;{\bf w}),$$
where
$|x|=1,|y|=2$, $|z_k|=m_k$ for $k=1,\ldots, n$, and ${\bf w}(x)={\bf w}(z_1)=\ldots
={\bf w}(z_n)=1$, ${\bf w}(y)=2$.

In particular, independently of the value of transgression homomorphisms we have
$$\Delta(M)\ge\ct(M)\ge [m_1+2m_2+\ldots+n m_n+(n+1)]+(m_1-1)\big(\dim M+1-\frac{m_1}{2}\big).$$
\end{theorem}
\begin{proof}
The main part of the argument consists in a spectral sequence computation showing that in
$H^*(M;\ZZ_p)$ there exists
a non-trivial product of the form $xy^{\frac{m_i-1}{2}}z_1\cdots\widehat{z_i}\cdots z_n$.

We will apply Leray-Serre spectral sequence with $\ZZ_p$-coefficients to the Borel fibration
$$\widetilde M\hookrightarrow E\ZZ_d\times_{\ZZ_d} \widetilde M \to B\ZZ_d$$
to compute $H^*(M;\ZZ_p)$. Let $p^k$, $k\ge 1$ be the biggest power of $p$ that divides $d$.
The $\ZZ_p$-cohomology algebra of $B\ZZ_d$ is well-known  (see Fact \ref{fact:Z_k cohomology of lens}):
$$H^*(B\ZZ_d;\ZZ_p)\cong H^*(B\ZZ_{p^k};\ZZ_p)\cong{\textstyle\bigwedge}(x)\otimes \ZZ_p[y],$$
with $|x|=1$, $|y|=2$ and $y=\beta_k(x)$, where $\beta_k$ is the $k$-th power Bockstein homomorphism
(see also \cite[Theorem 6.19]{McCleary2001} for a more complete reference). Since the action of $\ZZ_d$ on $H^*(\widetilde M;\ZZ_p)$
is trivial, the $E_2$ term of the $\ZZ_p$-cohomology Leray-Serre spectral sequence of the above
fibration is given as
$$E_2^{*,*}\cong \big({{\textstyle\bigwedge}(x)\otimes \ZZ_p[y]}\big)\otimes
{\textstyle\bigwedge} (u_1,\ldots,u_n),$$
where $|u_i|=m_i$. The even (respectively odd) dimensional elements of $\bigwedge(x)\otimes \ZZ_p[y]$ will be
denoted as $y^j$ (respectively $xy^j$). By the multiplicative properties of the spectral sequence it is
sufficient to study the value of the differentials on elements of the form $1\otimes u_i$.

If $r$ is odd, then $d_r(1\otimes u_i)=0$ by \cite[Proof of Theorem 1(3)]{DotzelSinghTripathi2001}.
Since we assumed that $m_1+m_2> m_n$, the only potentially non-zero homomorphism with domain $E_r^{0,m_i}$
is the transgression
$\tau_{m_i}=d_{m_i+1}\colon H^{m_i}(\widetilde M;\ZZ_p)\to H^{m_i+1}(B\ZZ_d;\ZZ_p).$
Let $i$ be the minimal integer, for which $\tau_{m_i}\ne 0$. Then
$\tau_{m_i}(u_i)=y^a\otimes 1$ for $a=(m_i+1)/2$. Consequently,
$d_{m_i+1}(y^j\otimes u_i)=y^{a+j}$ and $d_{m_i+1}(xy^j\otimes u_i)=xy^{a+j}$, which implies that
$E_r^{j,0}=0$ for $r,j>m_i+1$. We conclude that necessarily $\tau_{m_j}=0$ for $m_j> m_i$.
(Of course, $\tau_{m_j}$ may be non-trivial if $m_j=m_i$.
This is why we stated that there is a unique value $m_i$, for which $\tau_{m_i}$ is non-trivial.)

It is now easy to see that $x\otimes 1$, $y\otimes 1$ and $1\otimes u_j$ for $j\ne i$
are permanent cocycles in the spectral sequence, and that their product
$xy^{(m_i-1)/2}\otimes u_1\cdots \widehat{u_i}\cdots u_n\ne 0$ in $\mathrm{Tot}E_\infty^{*,*}$.
As usual, we may consider $x=x\otimes 1$ and $y=y\otimes 1$ as elements of $H^*(M;\ZZ_p)$ via the
edge homomorphism, while for elements $1\otimes u_j$ we may choose elements $z_j\in H^*(M;\ZZ_p)$,
$j\ne i$ that respectively project to $1\otimes u_j$. Then by naturality the product
$xy^{(m_i-1)/2}z_1\cdots \widehat{z_i}\cdots z_n$ projects to
$xy^{(m_i-1)/2}u_1\cdots \widehat{u_i}\cdots u_n$ and is thus non-zero.

Clearly $y=\beta_k(x)$ in $H^*(M;\ZZ_p)$, therefore $\swgt(y)=2$ and ${\bf w}$ satisfies the
requirements for a weight function, so by Theorem \ref{thm:swct estimate} we immediately obtain
the estimate
$$\ct(M)\ge \wct(xy^{\frac{m_i-1}{2}}z_1\cdots\widehat{z_i}\cdots z_n;{\bf w}).$$
It is possible to express the right-hand side in terms of dimensions $m_1,\ldots,m_n$ but the formula
is not particularly instructive, so we compute only one instance of it. It is clear that the value of the
lower bound that we computed increases with $i$, so for $i=1$ we obtain a lower bound that is valid
independently of which among the transgression homomorphism is non-trivial.

To begin, ${\bf w}(xy^{(m_1-1)/2}z_2\cdots z_n)=1+2\frac{m_1-1}{2}+(n-1)=n-1+m_1$, so we get
$$\wct(xy^{\frac{m_1-1}{2}}z_2\cdots z_n;{\bf w})=1+(n-1+m_1)+\sum_{i=1}^{n-1+m_1}
\max\bigg\{|\widehat{s}|\ \bigg|\ s\le xy^{(m_1-1)/2}z_2\cdots z_n,\ {\bf w}(s)\le i \bigg\}.$$
Since $\widetilde M$ is simply-connected, we have $3\le m_2\le\ldots\le m_n$, so the sub-products
of maximal dimension for $i=1,\ldots, n-1+m_i$ are
$$z_n,(z_{n-1}z_n),\ldots, (z_2\cdots z_n), (xz_2\cdots z_n),(yz_2\cdots z_n),(xyz_2\cdots z_n),\ldots,
(xy^{\frac{m_1-1}{2}}z_2\cdots z_n), $$
and the respective dimensions are
$$m_n,(m_{n-1}+m_n),\ldots,(m_2+\ldots+m_n),(m_2+\ldots+m_n)+1,(m_2+\ldots+m_n)+2,\ldots,(m_2+\ldots+m_n)+m_1
.$$
By summing up all contributions we obtain
$$\wct(xy^{\frac{m_1-1}{2}}z_2\cdots z_n;{\bf w})=[m_1+2m_2+\ldots+n m_n+(n+1)]+
(m_1-1)\big(\dim M+1-\frac{m_1}{2}\big).$$
\end{proof}

\begin{remark}
The form in which we stated the general lower estimate in the above theorem allows comparison with
lower bound for the covering type of a product of spheres given in
\cite[Proposition 3.1 and Theorem 3.5]{GovcMarzantowiczPavesic2019}:
$$\ct(S^{m_1}\times\cdots\times S^{m_n})\ge m_1+2m_2+\ldots+n m_n+(n+1).$$
It is evident that passing to the orbit space results in a considerable increase of the complexity
of the space, at least as it concerns its covering type.
\end{remark}

To illustrate our method with an explicit numerical example let us estimate the covering type of the
orbit space of a free cyclic action on a complex Stiefel manifold $V_3(\CC^7)$.

\begin{example}
Let $V_k(\CC^n)$ denote the Stiefel manifold of orthonormal $k$-frames in $\CC^n$. Its cohomology is
well-known and can be identified with the exterior algebra $\bigwedge(x_{(2(n-k)+1},\ldots,x_{2n-1})$
(cf. \cite[Example 5.G]{McCleary2001}). Moreover, the dimension of $V_k(\CC^n)$ is $k(2n-k)$.
One can easily check that the assumptions of Theorem \ref{thm:product of spheres}
are satisfied if $k\le \frac{n+2}{2}$.
So, for example, $V_3(\CC^7)$ is a 33-dimensional manifold and
by Theorem
\ref{thm:product of spheres} for every free action of $\ZZ_d$ on it we have
$\ct(V_3(\CC^7)/\ZZ_d)\ge 310$, therefore every triangulation of $V_3(\CC^7)/\ZZ_d$ has at least
310 vertices.

This general estimate can be improved if one knows, for example, that the transgression homomophism
$\tau_{13}$
in the spectral sequence of the Borel fibration is non-trivial. In that case Theorem
\ref{thm:product of spheres} yields a considerably better lower bound $\ct(V_3(\CC^7)/\ZZ_d)\ge 398$.
\end{example}

\subsection{Lie groups with cyclic fundamental group}\label{Lie groups}
An important family of examples whose covering type can be estimated with our methods are
non-simply connected Lie groups. Our results extend the results  of
\cite{DuanMarzantowiczZhao2021} where estimates of covering type  of all simple simply-connected
compact Lie groups were obtained based on rational cohomology rings and the methods of
\cite{GovcMarzantowiczPavesic2019}.

If $\mathcal{G}$ is a connected Lie group with a finite cyclic fundamental group,
then its universal covering $\widetilde{\mathcal{G}}$ is also a Lie group and the covering projection
$f\colon\widetilde{\mathcal{G}}\to \mathcal{G}$ is a group homomorphism. The kernel of $f$ is a normal
discrete subgroup
of $\widetilde{\mathcal{G}}$ and as such, it is contained in the centre of $\widetilde{\mathcal{G}}$.
Therefore $\mathcal{G}$ can be viewed
as a quotient of a simply-connected Lie group with respect to some central cyclic subgroup.
The relation between the cohomology rings of $\widetilde{\mathcal{G}}$ and of $\mathcal{G}$ was
first determined by Borel \cite{Borel1954} and later extended and improved by
Baum and Browder \cite{Baum-Browder}. We will see below that the description of the ring structure
is very similar to that of Theorem \ref{thm:product of spheres}, but a great advantage is
that there are no dimension restrictions, the system of coefficients is always orientable
and it is possible to determine precisely which transgression homomorphism is non-trivial.

The centres of compact connected simple Lie groups are well-known (see \cite{Mimura1995}): the only
instances with elements of odd order are groups of type A, i.e., the special unitary groups
$SU(n)$, whose centre is isomorphic to $\ZZ_{n}$, and the exceptional group $E_6$, whose centre
is isomorphic to $\ZZ_3$. In particular, every finite abelian group can be viewed as a central subgroup
of a product of special unitary groups.

We will first consider the Baum-Browder theorem which gives a very precise description of quotients
of the special unitary group $SU(n)$. The centre $C_n$ of $SU(n)$ consists of scalar matrices of
the form $\zeta I$ with $\zeta$ an $n$-th root of unity and is thus isomorphic to a cyclic group of
order $n$. As a consequence
every subgroup of $C_n$ is cyclic of order dividing $n$. Moreover, the cohomology ring of
$SU(n)$ is  $H^*(SU(n);\ZZ_p)\cong \bigwedge(x_2,\,\dots \,, x_n)$, the exterior algebra on $(n-1)$
variables of dimension $|x_i|=2i-1$ (cf. \cite[Section 2.2]{Mimura1995}). The cohomology ring of
a quotient of $SU(n)$ with respect to a subgroup of its centre is given by the following
theorem.

\begin{theorem}(Baum-Browder \cite[Thm I (7.12)]{Baum-Browder})
\label{thm: Baum-Browder}
Let $C=C_l$, $\vert C\vert  =l$,  be a subgroup of the centre $C_n \simeq \ZZ_n$ of  $SU(n)$, let $p$ be a
prime dividing $l$ and let $p^r$ be the highest power of $p$ dividing $n$.  Let $n=p^r \, n^\prime$, $l =p^t \, l^\prime$, where $n^\prime, $ $l^\prime$ are relatively prime with $p$.
If $p$ is odd,  or $p=2$ and $t\geq 2$,
$$H^*(SU(n)/C; \ZZ_p) = {\textstyle\bigwedge}(x)\otimes\ZZ_p[y]/(y^{p^r})
\otimes {\textstyle\bigwedge}(z_2, \, \dots, \, \,\widehat{z_{p^r}}\,, \dots \, ,\, z_n),$$
where $|x|=1,|y|=2,|z_i|=2i-1$ and $y=\beta(x)$.

\end{theorem}
Note that $H^*(SU(n)/C; \ZZ_p)$ does not depend on $l \mid  n$ but only on the $p$-power factor of $n$ provided $p\mid l$.
We present the estimate of $\ct(SU(n)/C)$, with $C=C_n$, since the argument for any $C_l\subset C_n$ is the same.

By combining Baum-Browder's description and   Theorem \ref{thm:swct estimate} we obtain the following
result.

\begin{theorem} \label{thm:quotients of SU(n)}
Let $C_n$ be a central subgroup of $SU(n)$ and let $k=p^r$ be a divisor of $n$ that is
a power of some odd prime, $r\geq 2$ if $p=2$.  Then
$$\ct(SU(n)/C)\ge \wct(xy^{k-1}z_2\cdots\widehat{z_k}\cdots z_n;{\bf w})=
\frac{n}{6}(4n^2+3n(4k-5)+5)-3(k-1)^2.$$
In particular, if $n$ is an odd prime power, $n=p^r$, or $p=2$ and $r\geq 2$, then
$$\ct(SU(n)/C_n)\ge \wct(xy^{n-1}z_2\cdots z_{n-1};{\bf w})=
\frac{8}{3}n(n-1)^2-\frac{n^2-25n+18}{6}.$$

The best (largest) estimate we have if the value of $k=p^r$ is maximal.
\end{theorem}
\begin{proof}

To estimate $\wct(xy^{k-1}z_2\cdots\widehat{z_k}\cdots z_n;{\bf w})$ we observe that
${\bf w}( xy^{k-1}z_2\cdots\widehat{z_k}\cdots z_n)=1+2(k-1)+(n-2)=n+2k-3$, and that for
$i=1,\ldots,n+2k-3$ the sub-products of maximal dimension are in increasing order of weight:
$$z_n,(z_nz_{n-1}),\ldots,(z_n\cdots z_{k+1}),\ldots,(z_n\cdots \widehat{z_k}\cdots z_2),\ldots,$$
$$(z_n\cdots \widehat{z_k}\cdots z_2)x,(z_n\cdots \widehat{z_k}\cdots z_2)y,\ldots,
(z_n\cdots \widehat{z_k}\cdots z_2)xy^{k-1}$$

Therefore the sum of the estimate can be calculated simply by counting how many times each of the classes $x$, $y$ and $z_i$ appears. Writing $s=xy^{p^r-1}z_2\cdots\widehat{z_{p^r}}\cdots z_n$, we
have:
\begin{align*}
\wct(s;\w)&=1+\w(s)+\sum_{k=1}^{\w(s)}|\widehat{s'}|\\
&=1+(n+2p^r-3)+(n+2p^r-3)|z_n|+\cdots+(3p^r-2)|z_{p^r+1}|\\
&\phantom{=1+(n+2p^r-3)}+(3p^r-3)|z_{p^r-1}|+\cdots +2p^r|z_2|+p^r|x|+p^r(p^r-1)|y|\\
&=1+(n+2p^r-3)+(n+2p^r-3)+\cdots +2p^r+p^r+2p^r(p^r-1)\\
&=\frac{2 n^3}{3}+2 n^2 p^r-\frac{5 n^2}{2}+\frac{5 n}{6}-3 p^{2 r}+6 p^r-3,\\
\end{align*}
as required.
The last part of statement follows from the lemma below.
\end{proof}

\begin{lemma}\label{Dejan}
The largest possible value of
\[
\wct(xy^{p^r-1}z_2\cdots\widehat{z_{p^r}}\cdots z_n;\w)
\]
is attained when $k=p^r$ is chosen to have the largest possible value.
\end{lemma}

\begin{proof}
Suppose $p_1^{r_1}<p_2^{r_2}$ are  prime powers dividing $n$. The inequality
\[
\frac{2 n^3}{3}+2 n^2 p_1^{r_1}-\frac{5 n^2}{2}+\frac{5 n}{6}-3 p_1^{2 r_1}+6 p_1^{r_1}-3<\frac{2 n^3}{3}+2 n^2 p_2^{r_2}-\frac{5 n^2}{2}+\frac{5 n}{6}-3 p_2^{2 r_2}+6 p_2^{r_2}-3
\]
is equivalent to
\[
(2 n^2+6-3 p_1^{r_1}) p_1^{r_1}<(2 n^2+6-3 p_2^{r_2}) p_2^{r_2},
\]
which is true, as the quadratic
\[
(2 n^2+6-3x)x
\]
is increasing for $x\leq\frac{n^2}{3}+1$.
\end{proof}

\begin{remark}
While the specific numerical formulas may not be particularly nice or elucidating, we may
still observe that triangulations of the quotient of the special unitary group by its centre
require at least around $\frac{8}{3}n^3$ vertices. For comparison, a triangulation of $SU(n)$
as estimated in \cite[Corollary 3.7]{GovcMarzantowiczPavesic2019} requires at least around
$\frac{2}{3}n^3$ vertices.
\end{remark}

If $C$ is a central subgroup of a general Lie group, then the analysis of various cases
becomes more complicated, but something can be still said, based on the following Borel's theorem.

\begin{theorem} (\cite[Proposition 10.3]{Borel1954})
Let $C$ be a cyclic central subgroup of a simply-connected Lie group $\widetilde{\mathcal{G}}$ and
let $p$ be an odd prime
that divides the order of $C$. Assume that the cohomology algebra $H^*(\widetilde{\mathcal{G}};\ZZ_p)$
is isomorphic to an exterior algebra on odd generators $z_1,\ldots,z_n$. Then
$$H^*(\widetilde{\mathcal{G}}/G;\ZZ_p)\cong {\textstyle\bigwedge}(x)\otimes \ZZ_p[y]/(y^s)\otimes
{\textstyle\bigwedge}(z_1,\ldots,\widehat{z_i},\ldots, z_n),$$
where $s=(|z_i|+1)/2$ is a power of $p$. In addition, the dimension
of the omitted generator $z_i$ is the first (and only) dimension
for which the transgression homomorphism is non-trivial.
\end{theorem}

The fact that $s$ must be a power of $p$ was not actually stated by Borel but it follows from the
description of monogenic Hopf algebras over $\ZZ_p$ (see \cite[Theorem III, 8.10]{Whitehead}).
Note that the assumption on the
cohomology of $\widetilde{\mathcal{G}}$ is valid for simple Lie groups with only a handful of
exceptions for
small primes (see Mimura \cite[Section 2.2]{Mimura1995}). Thus we obtain the following estimate for the
covering type of Lie groups whose fundamental group is cyclic.

\begin{theorem}\label{thm:Lie groups}
Let $\mathcal{G}$ be a connected Lie group with a cyclic fundamental group.
If for some odd prime $p$ dividing the order of $\pi_1(\mathcal{G})$,
$H^*(\widetilde{\mathcal{G}};\ZZ_p)$ is isomorphic to
an exterior algebra with odd-dimensional generators $z_1,\ldots,z_n$, then
$$\ct(\mathcal{G})\ge \wct(xy^{\frac{|z_i|-1}{2}}z_1\cdots\widehat{z_i}\cdots z_n;
{\bf w}),$$
where
$|x|=1,|y|=2$, and ${\bf w}(x)={\bf w}(z_1)=\ldots
={\bf w}(z_n)=1$, ${\bf w}(y)=2$.
\end{theorem}

\begin{example}
Let $\mathcal{G}$ be a compact, connected Lie group with $\pi_1(\mathcal{G})\cong\ZZ_{12}$ and which is
locally isomorphic
to $SU(3)\times SU(4)$. Therefore,
$H^*(\widetilde{\mathcal{G}};\ZZ_3)\cong \bigwedge(u_3,u_5)\otimes \bigwedge(v_3,v_5,v_7)$  (indices
denote
dimensions of the generators). Then without any prior knowledge of the transgression homomorphism in the
spectral sequence of the Borel fibration, we may apply Theorem \ref{thm:Lie groups} to obtain
$\ct(\mathcal{{G}})\ge \wct(x_1y_2u_5v_3 v_5 v_7;{\bf w})\ge 130$ (as usual, ${\bf w}(y):=\swgt(y)=2$,
while the
weight of other generators is set to 1), and so every triangulation of $\mathcal{G}$ requires at least 130
vertices.
\end{example}

\subsection{Orbit spaces of highly-connected manifolds}

Another large class of manifolds to which our theory applies are the even-dimensional manifolds
whose universal covering space is highly connected. To be specific, let $M$ be a closed $2n$-dimensional
manifold with $\pi_1(M)\cong\ZZ_d$ and such that $\widetilde M$ is homeomorphic to
$\sharp g(S^n\times S^n)$, the connected sum of $g$ copies of the product $S^n\times S^n$.
Observe that these already represent a large class of highly connected manifolds.
In fact, by the classification of $(n-1)$-connected $2n$-manifolds, and the solution of the
Kervaire invariant 1 problem, if $n\equiv 3,5,7 (\textrm{mod}\ 8)$ and $n\ne 15,31,63$,  then every
$(n-1)$-connected $2n$-manifold is homeomorphic to $\sharp g(S^n\times S^n)$ for a suitable $g$.
We will rely on recently published results by Su and Yang \cite{SuYang2021} and, to avoid being
too technical, we will only consider highly-connected $2n$-manifolds for $n$ odd and $\ZZ_d$-actions
for $d$ odd.

\begin{proposition} \label{prop:6-dim}
If $M$ is a 6-dimensional closed manifold with $\pi_1(M)\cong\ZZ_d$ for $d$ odd and $\pi_2(M)=0$, then
$\ct(M)\ge 24$.
\end{proposition}
\begin{proof}
By the assumptions, the universal covering $\widetilde M$ of $M$ is a 2-connected 6-dimensional
manifold. By the Wall's classification of highly-connected manifolds
$\widetilde M\approx \sharp g(S^n\times S^n)$, where $g=\beta_3(\widetilde M)/2$ is the genus of
$\widetilde M$. Su and Yang \cite[Theorem 1.3]{SuYang2021} proved that there is a unique
(up to conjugation) free $\ZZ_d$-action on a connected union of $S^3\times S^3$ and that the orbit space
(which is in our case homeomorphic to $M$) is of the form
$(L^3(d)\times S^3)\sharp \frac{g-1}{d}(S^3\times S^3)$. It is well-known that the projection
$\pi\colon M\to L^3(d)\times S^3$, which contracts all the other summands, induces an epimorphism in
cohomology. Then we take a prime $p$ dividing $d$ and recall from Proposition
\ref{prop:lens space x sphere} that there are
elements $x,y,z\in H^*(L^3(d)\times S^3;\ZZ_p)$ of degrees 1,2 and 3, respectively,
such that $xyz\ne 0$ and $\swgt(y)=2$. If we define $\bar x:=\pi^*(x), \bar y:=\pi^*(y),
\bar z:=\pi^*(z)$, then $\swgt(\overline y)=2$ and  moreover, $\bar x\bar y\bar z\ne 0$, because
$xyz$ is the generator of $H^6(L^3(d)\times S^3;\ZZ_p)$ and $\pi$ is of degree 1. Thus we may use
the result of Proposition \ref{prop:lens space x sphere} with $n=1$ and $m=3$ to obtain
the stated estimate.
\end{proof}

We have tried to take into account in the above estimate the genus of $M$ but without success, so this
remains an interesting open problem.

For $n\ge 4$ Su and Yang \cite[Theorem 1.5]{SuYang2021} were able to classify free $\ZZ_d$ actions on
connected sums of $S^n\times S^n$ under the assumption that the prime factors of $d$ are
relatively large. To be specific, for every $n\in\NN$ they define an odd prime $C(n)\in \NN$,
which depends on the exponents of stable homotopy groups and of the cokernel of the $J$-homomorphism
in a range of dimensions depending on $n$ (see \cite[p. 306]{SuYang2021} for a precise definition).
For instance, $C(n)=3$ for $n\le 7$ and $C(8)=C(9)=5$.

\begin{proposition}
Let $M$ be a $(4n+2)$-dimensional manifold with $\pi_1(M)\cong \ZZ_d$. Assume that the universal
covering of $M$  is homeomorphic to a connected sum of spaces $S^{2n+1}\times S^{2n+1}$
and that all prime factors of $d$ are bigger than $C(2n+1)$. Then
$\ct(M)\ge 6n^2+11n+7$.
\end{proposition}
\begin{proof}
Under the assumptions $M$ is homeomorphic to the orbit space $\sharp_g(S^{2n+1}\times S^{2n+1})/\ZZ_d$
of a suitable free $\ZZ_d$-action. Then we may use \cite[Theorem 1.5]{SuYang2021} and argue similarly as
in Proposition \ref{prop:6-dim} to conclude that for some prime $p$ there are elements
$\bar x, \bar y, \bar z\in H^*(M;\ZZ_p)$ of dimensions $1,2,(2n+1)$ respectively,
and such that $\bar x\bar y^n\bar z\ne 0$, $\swgt(\bar y)=2$. The stated estimate follows
by Proposition \ref{prop:lens space x sphere}.
\end{proof}

\subsection{Symplectic manifolds}

Another class of manifolds on which we may fruitfully apply our approach are the symplectic manifolds.
A basic (non-weighted) estimate given in \cite[Thm 2.13]{DuanMarzantowiczZhao2021}, states
that for $M$ a K\"{a}hler manifold or a closed symplectic manifold of (real) dimension $2m$ we have
$\ct(M)\ge (m+1)^2$.
In their approach toward the solution of Arnold's conjecture Oprea and Rudyak
\cite{Rudyak-Oprea} considered a special sub-class of symplectic manifolds.
Let $(M,\omega)$ be a $2m$-dimensional symplectic manifold with symplectic form $\omega$.
Then $M$ is said to be \emph{symplectically aspherical} if the class $\omega\in H^2(M;\RR)$
vanishes on the image of the Hurewicz homomorphism $h\colon \pi_2(M)\to H_2(M;\ZZ)\subset H_2(M;\RR)$.
Oprea and Rudyak \cite[Theorem 4.1]{Rudyak-Oprea} proved that every non-trivial class
$u\in H^2(X;\RR)$ that vanishes on  the image of the Hurewicz map $h\colon\pi_2(X)\to H_2(X)$
has $\swgt(u)=2$. As an immediate consequence, every symplectically aspherical manifold
$(M,\omega)$ has $\cat(M)=2m+1$, and thus by Theorem \cite[Theorem 2.2]{GovcMarzantowiczPavesic2019}
we have

\begin{proposition}
If $(M,\omega)$ is a $2m$-dimensional symplectically aspherical manifold, then
$$\Delta(M) \geq \ct (M)\geq (2m+1)(m+1).$$
\end{proposition}

Observe that the above estimate is only slightly better than that of Theorem \ref{thm:weighted estimate},
which gives $\ct(M)\ge \wct(\omega^n;{\bf w})=2m(m+1)+1$. This is not surprising as we are again in the
situation where the category of $M$ is maximal possible. However, if $(N,\tau)$ is a simply-connected
$2n$-dimensional symplectic manifold, then $\swgt(\tau)=1$ and $\cat(N)=n+1$.
The product $M\times N$ is also a symplectic
manifold with symplectic form $\omega\otimes 1+1\otimes \tau$ (which we abbreviate to $\omega+\tau$).
By direct computation $(\omega+\tau)^{2m+n}\ne 0$, while on the other side
$\cat(M\times N)\le \cat(M)+\cat(N)-1=2m+n+1$, therefore  $\cat(M\times N)=2m+n+1$, which is strictly
smaller than the dimension of $M\times N$. Indeed, $M\times N$ is symplectic, but it is not
symplectically aspherical and in that case  Theorem \ref{thm:weighted estimate} yields a much better estimate

\begin{proposition}
Let $(M,\omega)$ be a symplectically aspherical $2m$-dimensional symplectic manifold and let
$(N,\tau)$ be a $2n$-dimensional symplectic manifold. Then
$$\ct(M\times N)\ge 2m(2m+2n+1)+(n+1)^2.$$
\end{proposition}
\begin{proof}
To estimate the covering type of $M\times N$ we observe that
$\omega^m\tau^n\ne 0$ in $H^{2m+2n}(M\times N)$. By Theorem \ref{thm:swct estimate}
$$\ct(M\times N)\ge\wct(\omega^m\tau^n;{\bf w}),$$
where ${\bf w}(\omega)=\swgt(\omega)=2$ and ${\bf w}(\tau)=\swgt(\tau)=1$.
The value of $\wct(\omega^m\tau^n;{\bf w})$ is readily computed:
${\bf w}(\omega^m\tau^n)=2m+n$. For $k=1,\ldots,n$ the products
with weight at most $k$ and maximal dimension are $\tau,\tau^2,\ldots,\tau^n$, while for
$k=n+1,\ldots,2m+n$ the products
with weight at most $k$ and maximal dimension are of the form
$\tau^{n-1}\omega,\tau^n\omega,\tau^{n-1}\omega^2,\tau^{n}\omega^2,\ldots,\tau^n\omega^m$.
By summing up all contributions we obtain the estimate
$$\ct(M\times N)\ge 2m(2m+2n+1)+(n+1)^2.$$
\end{proof}

As a matter of comparison, if $N$ is simply-connected, then the LS-category estimate for
covering type yields  $\ct(M\times N)\ge (2m+n+1)(2m+n+2)/2$.

At the end we include an estimate of covering type of the Dold manifolds.

\begin{example}\label{ex: Dold}
For integers $r,\, s\geq 0$ a Dold manifold  $P(r, s)$  of dimension  $r +2 s$ is defined as
$$ P(r,s) = S^r \times  \CC P(s) / \sim $$
where $((x_1, \, \dots,\,  x_{r+1}),  [z_1, \, \dots \,, z_{s+1}]) \sim ((x_1, \, \dots, \,x_{r+1}) ,
[\bar{z}_1, \, \dots\,, \bar{z}_{s+1}])$, i.e.
$P(r,s)$   is the orbit space of a free involution on $S^r \times  \CC P(s)$ \cite{Dold}. The cohomology
mod $2$ of a Dold manifold is known (cf. \cite{Dold}):
$$ H^*(P(r,s); \ZZ_2)  \simeq  \ZZ_2[x, y]/ (x^{r+1}, y^{s+1}) $$
where $\deg (x) = 1$ and $\deg (y) = 2$. Unfortunately here $\swgt(y) = 1$ and our estimate of $
\wct(P(r,s))$ \newline reduces to that of \cite{GovcMarzantowiczPavesic2019} applied to the product
$x^r \, y^s$.

$$ \Delta(P(r,s)) \geq  \ct (P(r,s)) \geq 1 +  (r+2s) +\sum_{i=1}^r \, i \cdot 1 + \sum_{i=r+1}^{r+s} \, i
\cdot 2 = 1 + (s+r)(s+r+1) - \frac{r(r+1)}{2} \,.$$
\end{example}

\subsection{Final remarks}
We must emphasize that similarly as it was done in \cite{GovcMarzantowiczPavesic2020} and \cite{DuanMarzantowiczZhao2021} we can use the Lower Bound Theorem or its improved version  the Generalized
Lower Bound Theorem of \cite{Adiprasito} to estimate the number of simplices of  given dimension $i$ of a triangulation of manifold  $M$ (see \cite{KN} for a review of results). Indeed, this
theorem estimates from below the coordinates $f_i$, $i\geq 1$,  of  the vector ${\bf f}(m) = (f_0, \, f_1,\, \dots\,, f_d)$ by a formula which depends on $f_0$ (here $f_i$ is the number of
$i$-simplices of a PL-triangulation of the manifold). But in our notation  $f_0(M) =\Delta (M)$.

\section*{Acknowledgments}

The authors wish to express their thanks to Mahender Singh and Erg{\"u}n Yal\c{c}in for helpful
conversation on group actions on products of spheres and to John Oprea for his advice on
the properties of category weight.

\end{document}